\documentclass{amsart}
\usepackage{amssymb,latexsym}
\theoremstyle{plain}
\newtheorem{theorem}{Theorem}

\newtheorem{lemma}{Lemma}
\theoremstyle{definition}
\newtheorem{definition}{Definition}

\newtheorem{notation}{Notation}
\newtheorem{remark}{Remark}

\def\PP{\mathbb{P}^}

\date{}

\begin{document}

\title[]
{Minimal decomposition of binary forms with respect to tangential projections}
\author{E. Ballico}
\address{Dept. of Mathematics\\
 University of Trento\\
38123 Povo (TN), Italy}
\email{ballico@science.unitn.it}
\author{A. Bernardi}
\address{INRIA 2004 route des Lucioles, BP 93, 06902 Sophia Antipolis, France.}
\email{alessandra.bernardi@inria.fr}
\thanks{The authors were partially supported by CIRM--FBK, MIUR and GNSAGA of INdAM (Italy).}

\maketitle

\begin{abstract}
Let $C\subset \PP n$ be a rational normal curve and let $\ell_O:\PP {n+1}\dashrightarrow \PP n$ be any tangential projection form a point $O\in T_AC$ where $A\in C$. Hence
$X:= \ell _O(C)\subset \mathbb {P}^n$ is a linearly normal cuspidal curve with degree $n+1$. For any $P = \ell _O(B)$, $B\in \PP {n+1}$, the $X$-rank $r_X(P)$ of $P$ is the minimal cardinality of a set
$S\subset X$ whose linear span contains $P$. Here we describe $r_X(P)$ in terms of the schemes computing the $C$-rank or the border $C$-rank of $B$.
\end{abstract}

\section*{Introduction}
In many applications, like Biology and Statistics, it turns out to be useful to develop techniques for reducing the dimension of high-dimensional data (like Principal Component Analysis [PCA]) that can be encoded in a tensor. In many cases this tensor turns out to be symmetric and with many entries equal to zero. One of the main problem is to find a minimal decomposition of those tensors in terms of other tensors of the same structure but with the minimal number of entries as possible (in the literature this kind of problems are known  either as Structured Tensor Rank Decomposition in the Signal Processing  language --see e.g. \cite{c}-- or as CANDECOMP/PARAFAC in the Data Analysis context --see e.g. \cite{dmv}--).
We want to address these questions from an Algebraic Geometry point of view (we suggest  \cite{hps}  for a good description about the relation between Biology, Statistics and Algebraic Geometry on these kind of questions).

Let $V$ be a finite dimensional vector space defined over an algebraically closed field $K$ of characteristic zero. A symmetric tensor is an element $T\in S^dV$. Since the  space $S^dV$ is isomorphic to the vector space of homogeneous polynomial $K[x_1, \ldots , x_{\dim(V)}]_d$ of degree $d$  in $\dim(V)$ variables with the coefficients that take values over $K$, then one can translate the questions on symmetric tensors into questions on homogeneous polynomials.

In this paper we study the case of homogeneous polynomials of certain fixed degree $n+1$ in 2 variables having one coefficient equal to zero.  
\\
Assume for a moment to have fixed an order between the generators of $K[u,t]_{n+1}$ and to have given a corresponding coordinate system, say $\{x_0, \ldots , x_{n+1}\}$.
A binary form with the coefficient in the $i$-th position equal to zero  can be obtained by projecting a binary form to the hyperplane $H_i\subset K[u,t]_{n+1}$ identified by the equation $x_i=0$. We will focus on projections $\ell_O$ from a point $O\in \PP {}(K[u,t]_{n+1})\simeq \PP {n+1}$ to $\PP {}(H_i)\simeq \PP n$ that corresponds to tangential projections to the rational normal curve that is canonically embedded in $\PP {n+1}$. This will allow to relate the minimal decomposition of a binary form $p$ of degree $n+1$ as sum of $(n+1)$-th powers of linear forms $L_1^{n+1}, \ldots , L_r^{n+1}\in K[u,t]_{n+1}$, with the minimal decomposition of the projected $\ell_O(p)\in \PP {}(H_i)$ (that is a binary form of the same degree $n+1$ but with the $i$-th coefficient equal to zero) in terms of $\ell_O(L_1^{n+1}), \ldots , \ell_O(L_r^{n+1})$. Explicitly if $r$ is the minimum number of addenda that are required to write $p\in K[u,t]_{n+1}$ as
$$p=L_1^{n+1}+\cdots + L_r^{n+1}$$
then we will prove in Theorem \ref{e4} and in Theorem \ref{e3} that there is a dense  subset of $\PP {}(H_i)\simeq \PP n$ where $r$ is also the minimum number of addenda that are required to write $\ell_O(p)$ as follows:
$$\ell_O(p)=\ell_O(L_1^{n+1})+\cdots + \ell_O(L_r^{n+1}).$$
We will also describe which is the relation between the minimal decomposition of $p$ and the minimal decomposition of $\ell_O(p) $ out of this dense subset.
The minimal decomposition of a generic binary form of degree $n+1$ in terms of $(n+1)$-th powers of binary linear forms was firstly studied by J. J. Sylvester in \cite{Sy}, then formalized with an algorithm in \cite{CS} (see also \cite{bgi} for a more recent proof). 

What we want to study in this paper is obviously a very special case for applications (in the applications one often needs linear projections from large dimensional linear subspace) but we hope to give in this way some ideas for  further works on wider classes of analogous problems. In any case this kind of questions lead to a nice geometrical problem:
the computation of $X$-ranks with respect to a degree $n+1$ cuspidal linearly normal curve $X\subset \PP n$. We prove the following Theorem \ref{e4} and a less easy to state result (see Theorem \ref{e3}). In the statement of Theorem \ref{e4} and in Section \ref{secs2} we use the following definitions and notation.

\begin{notation}\label{not}
Let $C\subset \PP {n+1}$ be a smooth rational normal curve of degree $n+1$.
Fix $A\in C$. Let $2A$ denotes the degree $2$ effective divisor of $C$ with $A$ as its
reduction. The tangent line $T_AC$ is the line $\langle 2A \rangle$. Fix also a
point
$O\in T_AC\setminus \{A\}$ to be the center of the projection $\ell_O: \PP {n+1} \dashrightarrow \PP n$ that sends $C$ into a
curve $X:=\ell_O(C)\subset \PP n$. The curve  $X$ is a linearly normal curve of $\PP n$ with degree $n+1$,
arithmetic genus $1$ and the ordinary cusp $\ell_O(A)\in X\subset \PP n$ as its unique singular point.
\end{notation}

Let $Y\subset \PP N$ to be any non-degenerate projective variety.

\begin{definition}\label{rank}
The $Y$-rank $r_Y(P)$ of a point $P\in \PP N$ with respect to a non degenerate projective variety $Y$ is the minimum integer $\rho$ for which there exists a reduced $0$-dimensional subscheme $S\subset Y$ of degree $\rho$ whose linear span $\langle S\rangle$
contains $P$. 
\end{definition}

\begin{definition}\label{computerank}
Let $P\in \langle Y \rangle \simeq \PP N$ be a point of $Y$-rank equal to $\rho$. We say that a $0$-dimensional subscheme $S\subset Y$ computes the $Y$-rank of $P$ if it is reduced, of degree $\rho$ and $P\in \langle S \rangle$.
\end{definition}

Main results of this paper are Theorem \ref{e4} and Theorem \ref{e3} were we give a description of both the $X$-rank and the $X$-border rank of a point $P\in \mathbb{P}^n$ and we relate them with the $C$-rank and the $C$-border rank of its preimage via $\ell_O$.

\begin{theorem}\label{e4}
Fix integers $n,\rho$ such that $n \ge 3$ and $2\le \rho \le \lfloor (n+3)/2\rfloor$. Let $C \subset \mathbb {P}^{n+1}$ be a rational normal curve and let also $X:= \ell _O(C)\subset \PP n$ and $O\in T_AC\setminus \{A\}$ for a fixed $A\in C$ be as in Notation
\ref{not}. Fix $M\in \mathbb {P}^{n+1}\setminus \{O\}$
such that $r_C(M)=\rho$. Let $E\subset C$ be a finite set that computes the $C$-rank of $M$.
Set $P:= \ell_O(M)$. Then the following hold:

\begin{enumerate}
\item[(i)]\label{i1} 
If $2\rho \le n$, then $r_X(P)=\rho$ and $\ell _O(E)$ is the unique subset of $X$ computing $r_X(P)$.

\item[(ii)]\label{i2}
If $n+1 \le 2\rho \le n+2$, then
$\rho -1 \le r_X(P) \le \rho$.

\item[(iii)]\label{i3} 
If $n$ is odd and $2\rho =n+3$, then there is a non-empty open subset $\mathcal {U}$ of $\mathbb {P}^{n+1}$
such that $r_C(M)=\rho$ and $r_X(P) =\rho -1$ for all $M\in \mathcal {U}$.
\end{enumerate}
\end{theorem}

In Theorem \ref{e3} we take as $P$ a point $\ell _O(B)$ such that the border $C$-rank of $B$ is not computed by a reduced scheme, i.e. such that the $C$-rank of $B$ is strictly bigger than the $C$-rank of $B$.
\section{Preliminaries}\label{secs1}

We give here all the definitions and all the notation that we will need in the sequel. We can state them in a general setting even if we will use them in the very particular case of tangential projections of rational normal curves. So for this section we consider $Y\subset \PP N$ to be any non-degenerate projective variety.

\begin{definition}\label{sigma} Let $\sigma^0_s(Y)\subset \PP N$ denote the set of points $P\in \PP n$ of $Y$-rank less or equal than $s$.
The $s$-th secant variety $\sigma_s(Y)\subset \PP N$ is the Zariski closure of the set $\sigma^0_s(Y)$ of Definition \ref{sigma}.
\end{definition}

\begin{remark}\label{s0}
Definition \ref{sigma} gives the following obvious chain of containments:
$$Y=\sigma_1(Y)\subsetneqq \cdots \subsetneqq \sigma_{c-1}(Y)\subsetneqq \sigma_c(Y)=\PP N$$
for a certain integer $c>0$. If $Y$ is a non-degenerate curve, then $\dim (\sigma _s(Y)) = \min \{N,2s-1\}$ for all $s>0$ (\cite{a}, Remark 1.6) and hence $c:= \lfloor (N+2)/2\rfloor$.
\end{remark}

\begin{definition}\label{brank} Let $P\in \langle Y \rangle \subset \PP N$. The $Y$-border rank $br_Y(P)$ of $P$ is the minimum integer $w$ such that $P\in \sigma_w(Y)$.
\end{definition}

If $P\in \sigma_s(Y)\setminus \sigma^0_s(Y)$ then $r_Y(P)> s$. Definition \ref{sigma} gives $br_Y(P)\leq r_Y(P)$
for all $P\in \PP N$.

We borrow from \cite{bgl} the following result (we only need the case in which $Y$ is a rational normal curve of $\mathbb {P}^{n+1}$ with $2t \le n+2$; thus the case we use is a particular
case of \cite{bgl}, Lemma 2.1.5).

\begin{lemma}\label{uu}
Let $Y\subset \mathbb {P}^N$ be a smooth and non-degenerate subvariety of dimension at most $2$. Fix an integer $t\ge 2$
and assume $\dim (\langle Z\rangle )= \deg (Z)-1$ for every 0-dimensional
subscheme $Z\subset Y$ such that $\deg (Z)\le t$. Fix $P\in \mathbb {P}^N$.

\quad (i) $P\in \sigma _t(Y)$ if and only if there is a 0-dimensional scheme $Z\subset Y$ such that $\deg (Z) \le t$ and $P\in \langle Z\rangle$.

\quad (ii) $P\in \sigma _t(Y)\setminus \sigma _{t-1}(Y)$ if and only if $t$ is the first integer such that there is a 0-dimensional subscheme $Z\subset Y$ with $\deg (Z) =t$ and $P\in \langle Z\rangle$.

\end{lemma}

\begin{proof}
Since $Y$ is smooth and $\dim (Y) \le 2$, every 0-dimensional subscheme $A$ of $Y$ is smoothable, i.e. it is a flat limit of a family of unions of $\deg (A)$ distinct points
(\cite{f}).
As remarked in the proof of \cite{bgl}, Lemma 2.1.5, the assumption ``$\dim (\langle Z\rangle )= \deg (Z)-1$ for every 0-dimensional
scheme $Z\subset Y$ such that $\deg (Z)\le t$'' is sufficient to use \cite{bgi}, Proposition 11, and get part (i).

Part (ii) follows from part (i) applied to the integers $t$ and $t-1$.
\end{proof}

\begin{definition}\label{computeborder} Let $Y\subset \PP N$ be a smooth and non-degenerate variety of dimension at most $2$. Fix an integer $w\ge 2$
and assume $\dim ( \langle Z\rangle )= \deg (Z)-1$ for every 0-dimensional
subscheme $Z\subset Y$ such that $\deg (Z)\le w$. Let $P\in \sigma_w(Y)\setminus (\sigma_w^0(Y)\cup \sigma_{w-1}(Y))$, then, by Lemma \ref{uu}, there exists a non-reduced $0$-dimensional subscheme $W\subset Y$ such that $P\in \langle W \rangle$. We say that such a $W$ computes the $Y$-border rank of $P$.
\end{definition}

\begin{lemma}\label{e1}
Fix an integral and non-degenerate subvariety $Y \subset \mathbb {P}^{n+x}$, $n >0$, $x>0$, and a linear $(x-1)$-dimensional subspace $V \subset \mathbb {P}^{n+x}$ such that $V\cap Y = \emptyset$. Set $X:= \ell _V(Y)$.
Then
\begin{equation}\label{eqe1}
r_X(\ell _V(Q)) = \min _{P\in (\langle V\cup \{Q\}\rangle \setminus V)} r_Y(P) \  \hbox{ for all }Q\in \mathbb {P}^{n+x}\setminus V.
\end{equation}
\end{lemma}

\begin{proof} First of all let us prove the inequality ~``$\ge$~'' in (\ref{eqe1}).
Since $V\cap Y = \emptyset$, then obviously $\ell _V\vert Y$ is a finite morphism. Since $ \ell _V\vert Y : Y\to X$ is surjective, for each finite set of points $S\subset X$ we may fix another finite subset $S_V\subset Y$ 
such that $\ell _V(S_V)=S$ and $\sharp (S_V)=\sharp (S)$. Since $S_V \subseteq Y$, then $S_V\cap V =\emptyset$. Thus
the set $S\subset X$ turns out to be linearly independent if and only if $S_V$ is linearly
independent and $\langle S_V\rangle \cap V =\emptyset$. 
Now fix $Q\in \mathbb {P}^{n+x}\setminus V$ and take
$S\subset X$ computing $r_X(\ell _V(Q))$. Thus $\sharp (S) = r_X(\ell _V(Q))$ and $S$ is linearly independent by definition of a set that computes the $X$-rank of a point (see Definition \ref{computerank}). Since $S$ is linearly independent, the set $S_V$ is linearly
independent and $\langle S_V\rangle \cap V =\emptyset$. Now $\ell _V(Q)$ is an element of $\langle S\rangle$, then $\langle S_V\rangle \cap \langle V\cup \{Q\}\rangle \ne \emptyset$. Since $\langle S_V\rangle \cap V =\emptyset$, there
is a unique $P\in (\langle V\cup \{Q\}\rangle \setminus V)$ such that $\{P\} = \langle S_V\rangle \cap \langle V\cup \{Q\}\rangle$.
Since $S_V\subset Y$, we have $r_Y(P) \le \sharp (S_V) =\sharp (S) = r_X(\ell _V(Q))$.

To get the reverse inequality we may just quote Lemma 14 in \cite{bb} but since it is quite easy to be proved, we show here a shorter proof. Fix any $P\in (\langle V\cup \{Q\}\rangle \setminus V)$ and any $A\subset Y$ computing $r_Y(P)$.
Since $P\in (\langle V\cup \{Q\}\rangle \setminus V)$ we have $\ell _V(P) = \ell _V(Q)$. Since $\ell _V(P) \in \langle \ell _V(A)\rangle$, we have $r_X(\ell _V(Q)) \le r_Y(P)$.\end{proof}

\section{Theorems}\label{secs2}

We can now focus on tangential projections $X\subset \PP n$ of rational normal curves $C\subset \PP {n+1}$ for $n\geq 3$. We give both a description of the schemes that realize the $X$-border rank (Theorem \ref{e3})
and the $X$-rank (Theorem \ref{e4}) of a point $P\in \PP n$ with respect to a curve $X$ just described and the precise value
of the $X$-rank of such a point $P$ (except in the critical range $2w \ge n$ or $2\rho \ge n$, respectively). In Theorem \ref{e3} we give the $X$-rank of a point $P\in \PP n$ that is the image via
$\ell_O$ of a point $B\in \PP {n+1}$ whose $C$-border rank is smaller that its $C$-rank. In Theorem \ref{e4} the point $P\in
\PP n$ is the image of a point $M\in \PP {n+1}$ whose $C$-border rank is equal to its $C$-rank. Moreover we will
explain the relation between the schemes that compute $br_X(P)$ and $r_X(P)$ and the schemes that compute $br_C(B)$ and
$r_C(B)$ where $B\in \PP {n+1}$ is a point that is sent into $P\in \PP n$ by the tangential projection.

\begin{theorem}\label{e3} Fix integers $n, w$ such that $n \ge 3$ and $2w\le n+3$. Let $C \subset \mathbb {P}^{n+1}$ be a rational normal curve and let also $X:= \ell _O(C)\subset \PP n$ and $O\in T_AC\setminus \{A\}$ for a fixed $A\in C$ be as in Notation \ref{not}.
Fix $B\in \sigma_w(C)\setminus \sigma_w^0(C)\subset \PP {n+1}$ and set $P:= \ell _O(B)$.
Let $W \subset C$ be any degree $w$ subscheme which computes $br_C(B)$.

\begin{enumerate}

\item\label{1}
We have $O\in \langle W\rangle$ if and only if $A$ appears with multiplicity
at least $2$ in $W$.

\item\label{2} If either $A$ appears with multiplicity at least $3$ in $W$ or $A$ appears with multiplicity $2$ in $W$, but
$W\setminus 2A$ is not reduced, then $r_X(P)=n+3-w$. For every $w\ge 3$ this case occurs for some pair $(B,W)$.

\item\label{3} Assume that $A$ appears with multiplicity $2$ in $W$ and that $W\setminus 2A$ is reduced. Then both of
the following cases may occur:
\begin{itemize}
\item
either $r_X(P) = w-1$ and $\ell _O(W_{red}) $ computes $r_X(P)$ 
\item
or $w \ge 3$, $r_X(P) =w-2$
and $\ell _O(W\setminus 2A)$ computes $r_X(P)$.  
\end{itemize}

\item\label{4} Assume $2w \le n+1$, $O\notin \langle W\rangle$
and $A\notin W_{red}$. If $2w \le n-1$, then
$r_X(P) =n+1-w$. If $n \le 2w \le n+1$, then $n+1-w \le r_X(P) \le n+3-w$.

\item\label{5}
If $2w \le n$, $O\notin \langle W\rangle$
and $A\in W_{red}$, then $A$ appears with multiplicity $1$ in $W$ and
$r_X(P) =n+2-w$.

\end{enumerate}

\end{theorem}

\begin{proof} Since $C$ is a rational normal curve of $\mathbb {P}^{n+1}$, every 0-dimensional
subscheme $Z\subset C$ such that $\deg (Z) \le n+2$ is linearly independent, i.e. $\dim \langle Z\rangle =\deg (Z)-1$, with the usual conventions
$\deg (\emptyset )=0$, $\langle \emptyset \rangle =\emptyset$ and $\dim (\emptyset )=-1$. Thus if $Z_1, Z_2$ are 0-dimensional
subschemes of $C$ and $\deg (Z_1)+\deg (Z_2) \le n+2$, then $\langle Z_1\rangle \cap \langle Z_2 \rangle = \langle Z_1\cap Z_2\rangle$, where $Z_1\cap Z_2$ denote the scheme-theoretic intersection. 
First of all observe that since the subscheme $W\subset C$ computes $br_C(B)$ and $r_C(P) > br_C(P)$, then $W$ is not reduced.
Let us first prove the uniqueness of such a subscheme $W\subset C$.
Assume that $W_1\subset C$ is another such a subscheme. Hence $B\in \langle W\rangle \cap \langle W_1\rangle$. By Definition
\ref{computeborder} $B\notin \langle W'\rangle$ for any $W'\subsetneqq W$ and $\deg (W)+\deg (W_1)=2\deg(W)=2w\le
n+2$. Hence  $W_1\cap W = W$, i.e. $W_1=W$. 

Since we took $B\in \sigma_w(C)\setminus \sigma_w^0(C)$, then $r_C(B) = n+3-w$ (see \cite{bgi}, Theorem 23). Thus $r_C(P) \le n+3-w$.

\quad (a) Since $\deg (2A)+\deg (W) = 2+w\le n+3$, we have $\langle 2A\rangle \cap \langle W\rangle =\langle 2A\cap W\rangle$.
Thus $O\in \langle W\rangle$ if and only if $A$ appears in $W$ with multiplicity at least $2$.

Notice that
$O\in
\langle W\rangle$ if and only if
$\langle
\{O,B\}\rangle
\subseteq
\langle W\rangle$. We study now the case $O\in \langle W\rangle$. Fix any point $Q\in \langle \{O,B\}\rangle\setminus \{O\}$.
Since
$Q\in \langle W\rangle$, we have $br_C(Q) \le w$. Thus (by the so called Sylvester algorithm, see e.g. \cite{bgi}) either
$r_C(Q) = br_C(Q)$ or $r_C(Q) = n+3 -br_C(Q) \ge n+3-w = r_C(B)$. If the latter case occurs for all $Q\in \langle
\{O,B\}\rangle\setminus \{O\}$, then, by Lemma \ref{e1}, $r_X(\ell _O(B)) =n+3-w$.
\\
Assume the existence of $Q\in \langle \{O,B\}\rangle\setminus \{O\}$ such that $r_C(Q) =br_C(Q)$. Take $S_1\subset C$ computing $r_C(Q)$. Since $Q \in \langle W\rangle$,
the proof of the uniqueness of $W$ gives $S_1\subseteq W$. Since $W$ is not reduced, then $S_1\subsetneqq W$. Since
$Q\ne O$, and $B \notin \langle S_1\rangle$, $Q$ is the only point of the line
$\langle \{O,B\}\rangle$ contained in $\langle S_1\rangle$. Since $O\in \langle 2A\rangle$, we get $\langle \{O,B\}\rangle \subseteq \langle 2A\cup S_1\rangle$. Thus the uniqueness of
$W$ gives $W \subseteq S_1\cup 2A$. Since $S_1$ is reduced, $S_1\subseteq  W$ and $W$ is not reduced, we get that $A$
appears with multiplicity $2$ in $W$ and $W = S_1\cup 2A$. 
If $A\notin S_1$, then $\sharp (S_1) = w-2$. If $A\in S_1$, then $\sharp (S_1) = w-1$
and $W_{red} = S_1$.
\\
We want to show that both cases
occur for certain points $B$ if $w \ge 3$ and we also want to describe all points $B$ for which they occur.
\\
If $W =2A$, then $P=\ell _O(B)\in X$ and hence $r_X(P) =1$. 
\\
Now assume $W\ne 2A$, i.e. $w \ge 3$. Take any $S_2\subset C$ such that $\sharp (S_2) =w-2$ and $A\notin S_2$. Set $W:= S_2\cup 2A$.
Since $w \le n+2$, we saw that $W$ is linearly independent. Set $\Sigma := \langle S_2\cup \{O\}\rangle$. Since $O\ne A$, there is a non-empty
open subset $\Omega$ of the $(w-2)$-dimensional projective space $\Sigma$ such that if $B_1\in \Omega$, then $B_1\notin \langle W'\rangle$ for all $W'\subsetneqq W$. If $B_1\in \Omega$, then $\ell _O(B_1)\in
\langle \ell _O(S_2)\rangle$. If $B_2\in \langle W\rangle \setminus \Sigma$ and $B_2\notin \langle W'\rangle$ for all $W'\subsetneqq B$, then $\ell _O(B_2)\in \langle
\ell _O(\{A\}\cup S_2)\rangle$ and $\ell _O(B_2)\notin \langle
\ell _O(S_2)\rangle$. We saw that $r_X(\ell _O(B_1)) = w-2$ and $r_X(\ell _O(B_2)) = w-1$.

Now we check that for every integer $w$ such that $3 \le w \le (n+3)/2$ we may find $(W,B)$ such that $O\in \langle W\rangle$ and $r_X(\ell _O(B)) = n+3-w$. We just saw that this is the case
for all $W$ containing $A$ with multiplicity at least $3$ and for all $B\in \langle W\rangle$ such that $B\notin \langle W'\rangle$ for any $W'\subsetneqq W$. We saw
that for every $w \le (n+3)/2$ and any degree $w$ scheme $Z\subset C$, the scheme $Z$ computes $br_C(D)$ for all $D\in \langle
Z\rangle$ such that $D\notin \langle Z'\rangle$ for any $Z'\subsetneqq Z$.

\quad (b) Here we assume $2w \le n+1$, $O\notin \langle W \rangle$ and $A\notin W_{red}$. In this case the dimension of
$\langle 2A\cup W\rangle
$ is $w+1$ because $\deg (2A\cup W)=2+w \le n$.  Hence $\langle 2A\rangle\cap \langle W\rangle =\emptyset$
and $\langle 2A\rangle \cap \langle \{A\}\cup W\rangle =\{A\}$. Since $O\in \langle 2A\rangle$ and $O\ne A$, we get that $O\notin \langle \{A\}\cup W\rangle$.
Fix any point
$Q\in
\langle
\{O,B\}\rangle
\setminus
\{O\}$. Since $B\in \langle W\rangle$ and $O\in \langle 2A\rangle$, we have $Q\in \langle 2A\cup W\rangle$. Thus, by  \cite{bgi}, Proposition 11,
$\eta := br_C(Q) \le w+2$. Let $E\subset C$ be any scheme computing $br_C(Q)$. The scheme $E$ is unique if $2\eta \le n+2$. Since $B\in \langle W\rangle \subset \langle \{A\}\cup W\rangle$,
$O\in \langle 2A\rangle$, $O\notin \langle \{A\}\cup W\rangle$, $Q\ne A$ and $\langle \{A\} \cup W\rangle \cap \langle 2A\rangle =\{A\}$, then $Q\notin \langle \{A\} \cup W\rangle$.
\\
Assume the existence of a proper subscheme $G\subsetneqq 2A\cup W$ such that $Q\in \langle G\rangle$. Since $Q\notin \langle \{A\} \cup W\rangle$ there is $G_1\subsetneqq W$ such
that $Q\in \langle 2A\cup G_1\rangle$. Since $O\in \langle 2A\rangle$, we get $B\in \langle 2A\cup G_1\rangle$. Since $\deg (2A\cup G_1) +\deg (W) \le 2+w-1+w \le n+2$, we get
$\langle 2A\cup G_1\rangle \cap \langle W\rangle = \langle G_1\rangle$, contradicting the assumption $br_C(B)=w$.
Thus
there is no proper subset $G$ of $2A\cup W$ such that $Q\in \langle G\rangle$. 

\quad (b1) Here we assume $2w \le n-2$. Since $\deg (E)+\deg (2A\cup W) \le \eta +w+2 \le 2w+4 \le n+2$, we have $\langle E\rangle \cap \langle 2A\cup W\rangle
= \langle E\cap (2A\cup W)\rangle$. Since $Q\in \langle E\rangle \cap \langle 2A\cup W\rangle$ and $Q\notin \langle G\rangle$ for any proper subscheme $G$ of $2A\cup W$, we
get $E = 2A\cup W$. Thus $\eta = w+2$ and $br_C(Q)$ is computed by an unreduced scheme. Thus $r_C(Q) =n+1-w$ (Sylvester). Lemma \ref{e1} gives $r_X(P) = n+1-w$.

\quad (b2) Here we assume $2w = n-1$. First assume $\eta \le w+1$. Since $\deg (E)+\deg (2A\cup W) \le \eta +w+2 \le n+2$ and
$Q\in \langle E\rangle \cap \langle 2A+W\rangle$, we get $E\supseteq 2A\cup W$, contradiction. Thus $\eta = w+2$. If $E$ is not reduced, then we get
$r_C(Q) = n+1-w$ (Sylvester). If $E$ is reduced, then $r_C(Q) = w+2 =n+1-w$. Hence
in both cases Lemma \ref{e1} gives $r_X(P) = n+1-w$.

\quad (b3) Here we assume $n \le 2w \le n+1$. As above we get $\eta \ge n+1-w$. Hence $r_C(Q) \ge n+1-w$. Lemma \ref{eqe1} gives $r_X(P) \ge n+1-w$.

\quad (c) Assume $2w\le n$ and $A\in W_{red}$. If $A$ appears with multiplicity at least $2$ in $W$, then $O\in \langle W\rangle$. This case was considered in step (a). Now assume that $A$
appears with multiplicity $1$ in $W$.
Then $A\in \langle W\rangle$ and $\deg (W\cup 2A)=w+1$. Set $W_1:= W\setminus \{A\}$ and $W_2:= W_1\cup 2A$. Thus $\deg (W_2)
=\deg (W_1)+2 =w+1$. By step (a) we have $O\notin \langle W\rangle$. Fix any
$Q\in
\langle
\{O,B\}\rangle
\setminus
\{O,B\}$. Since $B\in \langle W\rangle$ and $O\notin \langle W\rangle$, then $Q\notin \langle W\rangle$.
Since $O\in \langle 2A\rangle$, then $Q\in \langle W_2\rangle$. Thus $br_C(Q) \le w+1$. Since $2(w+1) \le n+2$, we also know that $br_C(Q)$
is computed by a unique scheme $\Gamma$ and that $\Gamma \subseteq W_2$. Since $B\in \langle 2A\cup \Gamma \rangle$, we also
have $W \subseteq \Gamma \cup 2A$. Hence either $\Gamma = W_2$ or $\Gamma = W$ or $\Gamma = W_1$. Since $Q\notin \langle W\rangle$, we have
$\Gamma = W_2$. Thus $br_C(Q) =w+1$ and $b_C(Q)$ is computed by an unreduced subscheme. Thus $r_C(Q) =n+3 -br_C(Q) =n+2-w$
(\cite{bgi}, Theorem 23). Lemma \ref{e1} gives $r_X(P) =n+2-w$.\end{proof}

\vspace{0.3cm}

\qquad {\emph {Proof of Theorem \ref{e4}.}} Since $2\rho \le n+2$, the scheme $E$ is unique (see \cite{bgl}, Theorem 1.4.2, for Veronese embeddings of any projective space). Let us first  check that  $O\notin \langle E\rangle$.
Assume $O\in \langle E\rangle$. Thus $O\in \langle 2A\rangle \cap
\langle E\rangle$. Since
$\deg (2A)+\deg (E)=2+\rho \le n+2$, we get $\langle 2A\rangle \cap \langle E\rangle =\langle \{A\}\cap E\rangle$. Since $O\ne A$ and $O\in \langle 2A\rangle$, we get a contradiction. Therefore $O\notin \langle E\rangle$.

Since $\ell _O\vert C$ is injective, $\sharp (\ell _O(E))=\rho$.
Obviously, $\ell _O(B)=P\in \langle \ell _O(E)\rangle$. Thus $r_X(P) \le \rho$. But we have just proved that $O\notin \langle E\rangle$, i.e.
$\dim (\langle \ell _O(E)\rangle )=\rho$.

\quad (a) Here we assume $2\rho \le n$. Take $S\subset X$ computing $r_X(\ell _O(M))$. 
In this case it is sufficient to prove that $S=\ell _O(E)$. 
Since $\ell _O\vert C$ is injective, there
is a unique $S'\subset C$ such that $\ell _O(S')=S$. Since $P=\ell_O(M)\in \langle S\rangle$, we have
$M \in \langle \{O\}\cup S'\rangle \subset \langle 2A \cup S'\rangle$. Thus $M \in \langle 2A \cup S'\rangle \cap \langle
E\rangle$.
Since
$\deg (2A\cup S') + \deg (E) \le 2+2\rho \le n+2$, 
the scheme $2A\cup S'\cup E$ is linearly
independent. Thus $\langle 2A \cup S'\rangle \cap \langle
E\rangle$ is the linear span of the scheme-theoretic intersection $(2A\cup S')\cap E$. Since $E$
is reduced  and $M\notin \langle E'\rangle$ for any $E'\subsetneqq E$, we get that either $S' = E$ or $S'\cup \{A\} =E$. If $A\notin E$, then we get
$S'=E$, as wanted. 
\\
Now assume $A\in E$. If $S'=E$, then we are done. Hence we may also assume that $S' \ne E$, i.e. $S' = E\setminus \{A\}$. Since $M\in \langle E\rangle \setminus \langle S'\rangle$,
we have $\langle E\rangle = \langle \{M\}\cup S'\rangle$. Thus $O\notin \langle \{M\}\cup S'\rangle$. Thus $\ell _O(M) \notin \langle \ell _O(S')\rangle$, contradiction.

\quad (b) Here we assume $n+1 \le 2\rho \le n+2$.  Assume $r_X(P) \le \rho -2$ and take $S'\subset C$ such that $\sharp (S') = r_X(P)$ and $\ell _O(S')$ computes $r_X(P)$. Since $\deg (2A\cup S'\cup E) \le n+2$,
as in step (a) we get that $\langle 2A \cup S'\rangle \cap \langle
E\rangle$ is the linear span of the scheme-theoretic intersection $(2A\cup S')\cap E$. Since $O\in \langle 2A\rangle$ and $M\in \langle \{O\}\cup S'\rangle \cap \langle E\rangle$,
while $M\notin \langle E'\rangle$ for any $E'\subsetneqq E$, we get a contradiction.

\quad (c) Assume $n$ odd and $2\rho =n+3$. A general $P_1\in \mathbb {P}^{n+1}$ satisfies $r_C(P_1) =br_C(P_1) = (n+3)/2$. A
general
$P'\in \mathbb {P}^n$ satisfies $r_X(P') =br _X(P') = (n+1)/2$. A general $P'\in \mathbb {P}^n$ is of the form $\ell _O(P_1)$ with $P_1$ general in $\mathbb {P}^{n+1}$.\qed

\providecommand{\bysame}{\leavevmode\hbox to3em{\hrulefill}\thinspace}

\end{document}